\documentclass[letterpaper,10pt]{amsart}

\usepackage[all]{xy}                        %

\CompileMatrices                            

\UseTips                                    

\input xypic
\usepackage[bookmarks=true]{hyperref}       

\usepackage{amssymb,latexsym,amsmath,amscd}
\usepackage{xspace}
\usepackage{color}
\usepackage{kpfonts}
\usepackage{graphicx}

\reversemarginpar

\theoremstyle{plain}
\newtheorem{theorem}{Theorem}[section]
\newtheorem*{theorem*}{Theorem}

\newtheorem{lemma}[theorem]{Lemma}

\theoremstyle{definition}
\newtheorem{definition}[theorem]{Definition}

\newtheorem{remark}[theorem]{Remark}
\newtheorem{example}[theorem]{Example}

\newcommand{\enm}[1]{\ensuremath{#1}}          %
\newcommand{\op}[1]{\operatorname{#1}}
\newcommand{\cal}[1]{\mathcal{#1}}

\newcommand{\CC}{\enm{\mathbb{C}}}

\newcommand{\NN}{\enm{\mathbb{N}}}

\newcommand{\ZZ}{\enm{\mathbb{Z}}}

\newcommand{\PP}{\enm{\mathbb{P}}}

\newcommand{\Aa}{\enm{\cal{A}}}
\newcommand{\Bb}{\enm{\cal{B}}}

\newcommand{\Ee}{\enm{\cal{E}}}
\newcommand{\Ff}{\enm{\cal{F}}}

\newcommand{\Ll}{\enm{\cal{L}}}

\newcommand{\Oo}{\enm{\cal{O}}}

\newcommand{\Qq}{\enm{\cal{Q}}}

\newcommand{\Vv}{\enm{\cal{V}}}

\renewcommand{\phi}{\varphi}
\renewcommand{\theta}{\vartheta}
\renewcommand{\epsilon}{\varepsilon}

\newcommand{\Pic}{\op{Pic}}

\newcommand{\Ext}{\op{Ext}}

\renewcommand{\to}[1][]{\xrightarrow{\ #1\ }}

\newcommand{\old}[1]{}

\begin{document}

\title[Ulrich bundles on smooth projective varieties of minimal degree]{Ulrich bundles on smooth projective varieties of minimal degree}

\author{M. Aprodu, S. Huh, F. Malaspina and J. Pons-Llopis}

\address{Faculty of Mathematics and Computer Science, University of Bucharest, 14 Academiei
Street, 010014 Bucharest, Romania \& "Simion Stoilow" Institute of Mathematics of the Romanian Academy, P. O. Box 1-764, 014700 Bucharest, Romania}
\email{marian.aprodu@fmi.unibuc.ro \& marian.aprodu@imar.ro}

\address{Sungkyunkwan University, Suwon 440-746, Korea}
\email{sukmoonh@skku.edu}

\address{Politecnico di Torino, Corso Duca degli Abruzzi 24, 10129 Torino, Italy}
\email{francesco.malaspina@polito.it}

\address{Department of Mathematics, Kyoto University, Kyoto, Japan}
\email{ponsllopis@gmail.com}

\thanks{The first author was partly supported by a UEFISCDI grant. The second author is supported by Basic Science Research Program 2015-037157 through NRF funded by MEST and the National Research Foundation of Korea(KRF) 2016R1A5A1008055 grant funded by the Korea government(MSIP). The third author is supported by the framework of PRIN 2010/11 \lq Geometria delle variet\`a algebriche\rq, cofinanced by MIUR and GNSAGA of INDAM (Italy). The forth author is supported by a FY2015 JSPS Postdoctoral Fellowship}


\begin{abstract}
We classify the Ulrich vector bundles of arbitrary rank on smooth projective varieties of minimal degree. In the process, we prove the stability of the sheaves of relative differentials on rational scrolls.
\end{abstract}

\maketitle


\section{Introduction}
There has been increasing interest on the classification of arithmetically Cohen-Macaulay (for short aCM) sheaves on various projective varieties, which is important in a sense that the aCM sheaves are considered to give a measurement of complexity of the underlying space. A special type of aCM sheaves, called the Ulrich sheaves, are the ones achieving the maximum possible minimal number of generators. It is conjectured in \cite{ES} that any variety supports an Ulrich sheaf. But the conjecture has been checked only for a few varieties, e.g. in case of surfaces, del Pezzo surfaces, rational normal scrolls, rational aCM surfaces in $\PP^4$, ruled surfaces  and so on; see  \cite{CKM, MR, MP, ACM}. Although there are some occasions where the classification problem of Ulrich bundles of special type is done as in \cite
{CCHMW, CM1, CM2}, the completion of classification problem is difficult in usual.

In this article we pay our attention to the smooth projective varieties of minimal degree; smooth quadric hypersurfaces, the Veronese surface in $\PP^5$ and rational normal scrolls $S=S(a_0, \ldots, a_n)$ associated to rank $n+1$ bundle $\oplus_{i=0}^n \Oo_{\PP^1}(a_i)$ on $\PP^1$. We choose suitable full exceptional collections of those varieties to apply the Beilinson spectral sequence and characterize the Ulrich bundles supporting the varieties. In \cite{MR} the representation type of $S$ is determined by considering a certain type of family of Ulrich sheaves and it motivates the study in this article. 

In case of the Veronese surface and smooth quadric hypersurfaces, the classical result that the only indecomposable Ulrich bundles are $\Omega_{\PP^2}^1(3)$ and the spinor bundles, respectively, can be recovered; see Examples \ref{ver} and \ref{qua}. The main result of this article is the characterization of Ulrich bundles on smooth rational normal scrolls as the bundles admitting a special type of filtration; see Theorem \ref{volon}. As a direct consequence, we observe that the moduli spaces of Ulrich bundles are zero-dimensional; see remark \ref{con}. Since any non-linearly normal smooth variety of almost minimal degree is given as the image of a linear projection of a smooth variety of minimal degree (and the projection is an isomorphism in this case), the classification of Ulrich bundles on the non-linearly normal smooth varieties of almost minimal degree can be also given; see \cite{Park}. Indeed, being Ulrich only depends on the chosen polarization $(X,\Ll)$, not on taking the embedding on the complete linear system $H^0(X,\Ll)$.

Here we summarize the structure of this article. In Section \ref{sec2} we introduce the definition of Ulrich sheaves and several notions in derived category of coherent sheaves to understand the Beilinson spectral sequence. Then we investigate the indecomposable Ulrich bundles on the Veronese surface and smooth quadric hypersurfaces. In Section \ref{sec3} we collect several technical results about several bundles on rational normal scrolls; see Lemmas \ref{bb} and \ref{llem}. Then we fix a certain full exceptional collection of the derived categories in Example \ref{fec}, based on which we apply the Beilinson spectral sequence to characterize the Ulrich bundles in Theorem \ref{volon}.


\section{Preliminaries}\label{sec2}
Throughout the article our base field is the filed of complex numbers $\CC$.

\begin{definition}
A coherent sheaf $\Ee$ on a projective variety $X$ with a fixed ample line bundle $\Oo_X(1)$ is called {\it arithmetically Cohen-Macaulay} (for short, aCM) if it is locally Cohen-Macaulay and $H^i(\Ee(t))=0$ for all $t\in \ZZ$ and $i=1, \ldots, \dim (X)-1$.
\end{definition}



\begin{definition}
For an {\it initialized} coherent sheaf $\Ee$ on $X$, i.e. $h^0(\Ee(-1))=0$ but $h^0(\Ee)\ne 0$, we say that $\Ee$ is an {\it Ulrich sheaf} if it is aCM and $h^0(\Ee)=\deg (X)\mathrm{rank}(\Ee)$.
\end{definition}

Given a smooth projective variety $X$, let $D^b(X)$ be the the bounded derived category of coherent sheaves on $X$. An object $E \in D^b(X)$ is called {\it exceptional} if $\Ext^\bullet(E,E) = \mathbb C$.
A set of exceptional objects $\langle E_0, \ldots, E_n\rangle$ is called an {\it exceptional collection} if $\Ext^\bullet(E_i,E_j) = 0$ for $i > j$. An exceptional collection is said to be {\it full} when $\Ext^\bullet(E_i,A) = 0$ for all $i$ implies $A = 0$, or equivalently when $\Ext^\bullet(A, E_i) = 0$ does the same.

\begin{definition}\label{def:mutation}
Let $E$ be an exceptional object in $D^b(X)$.
Then there are functors $\mathbb L_{E}$ and $\mathbb R_{E}$ fitting in distinguished triangles
$$
\mathbb L_{E}(T) 		\to	 \Ext^\bullet(E,T) \otimes E 	\to	 T 		 \to	 \mathbb L_{E}(T)[1]
$$
$$
\mathbb R_{E}(T)[-1]	 \to 	 T 		 \to	 \Ext^\bullet(T,E)^* \otimes E	 \to	 \mathbb R_{E}(T)	
$$
The functors $\mathbb L_{E}$ and $\mathbb R_{E}$ are called respectively the \emph{left} and \emph{right mutation functor}.
\end{definition}


The collections given by
\begin{align*}
E_i^{\vee} &= \mathbb L_{E_0} \mathbb L_{E_1} \dots \mathbb L_{E_{n-i-1}} E_{n-i};\\
^\vee E_i &= \mathbb R_{E_n} \mathbb R_{E_{n-1}} \dots \mathbb R_{E_{n-i+1}} E_{n-i},
\end{align*}
are again full and exceptional and are called the \emph{right} and \emph{left dual} collections. The dual collections are characterized by the following property; see \cite[Section 2.6]{GO}.
\begin{equation}\label{eq:dual characterization}
\Ext^k(^\vee E_i, E_j) = \Ext^k(E_i, E_j^\vee) = \left\{
\begin{array}{cc}
\mathbb C & \textrm{\quad if $i+j = n$ and $i = k $} \\
0 & \textrm{\quad otherwise}
\end{array}
\right.
\end{equation}

\begin{theorem}[Beilinson spectral sequence]\label{thm:Beilinson}
Let $X$ be a smooth projective variety and with a full exceptional collection $\langle E_0, \ldots, E_n\rangle$ of objects for $D^b(X)$. Then for any $A$ in $D^b(X)$ there is a spectral sequence 
with the $E_1$-term
\[
E_1^{p,q} =\bigoplus_{r+s=q} \Ext^{n+r}(E_{n-p}, A) \otimes \mathcal H^s(E_p^\vee )
\]
which is functorial in $A$ and converges to $\mathcal H^{p+q}(A)$.
\end{theorem}

The statement and proof of Theorem \ref{thm:Beilinson} can be found both in  \cite[Corollary 3.3.2]{RU}, in \cite[Section 2.7.3]{GO} and in \cite[Theorem 2.1.14]{BO}. 


Let us assume next that the full exceptional collection  $\langle E_0, \ldots, E_n\rangle$ contains only pure objects of type $E_i=\mathcal E_i^*[-k_i]$ with $\mathcal E_i$ a vector bundle for each $i$, and moreover the right dual collection $\langle E_0^\vee, \ldots, E_n^\vee\rangle$ consists of coherent sheaves. Then the Beilinson spectral sequence is much simpler since 
\[
E_1^{p,q}=\Ext^{n+q}(E_{n-p}, A) \otimes E_p^\vee=H^{n+q+k_{n-p}}(\mathcal E_{n-p}\otimes A)\otimes E_p^\vee.
\]

Note however that the grading in this spectral sequence applied for the projective space is slightly different from the grading of the usual Beilison spectral sequence, due to the existence of shifts by $n$ in the index $p,q$. Indeed, the $E_1$-terms of the usual spectral sequence are $H^q(A(p))\otimes \Omega^{-p}(-p)$ which are zero for positive $p$. To restore the order, one needs to change slightly the gradings of the spectral sequence from Theorem \ref{thm:Beilinson}. If we replace, in the expression 
\[
E_1^{u,v} = \mathrm{Ext}^{v}(E_{-u},A) \otimes E_{n+u}^\vee=
H^{v+k_{-u}}(\mathcal E_{-u}\otimes A) \otimes \mathcal F_{-u}
\]
$u=-n+p$ and $v=n+q$ so that the fourth quadrant is mapped to the second quadrant, we obtain the following version of the Beilinson spectral sequence:



\begin{theorem}\label{use}
Let $X$ be a smooth projective variety with a full exceptional collection 
$\langle E_0, \ldots, E_n\rangle$ 
where $E_i=\mathcal E_i^*[-k_i]$ with each $\mathcal E_i$ a vector bundle and $(k_0, \ldots, k_n)\in \ZZ^{\oplus n+1}$ such that there exists a sequence $\langle F_n=\mathcal F_n, \ldots, F_0=\mathcal F_0\rangle$ of vector bundles satisfying
\begin{equation}\label{order}
\mathrm{Ext}^k(E_i,F_j)=H^{k+k_i}( \mathcal E_i\otimes \mathcal F_j) =  \left\{
\begin{array}{cc}
\mathbb C & \textrm{\quad if $i=j=k$} \\
0 & \textrm{\quad otherwise}
\end{array}
\right.
\end{equation}
i.e. the collection $\langle F_n, \ldots, F_0\rangle$ labelled in the reverse order is the right dual collection of $\langle E_0, \ldots, E_n\rangle$.
Then for any coherent sheaf $A$ on $X$ there is a spectral sequence in the square $-n\leq p\leq 0$, $0\leq q\leq n$  with the $E_1$-term
\[
E_1^{p,q} = \mathrm{Ext}^{q}(E_{-p},A) \otimes F_{-p}=
H^{q+k_{-p}}(\mathcal E_{-p}\otimes A) \otimes \mathcal F_{-p}
\]
which is functorial in $A$ and converges to 
\begin{equation}
E_{\infty}^{p,q}= \left\{
\begin{array}{cc}
A & \textrm{\quad if $p+q=0$} \\
0 & \textrm{\quad otherwise.}
\end{array}
\right.
\end{equation}
\end{theorem}

\begin{example}\label{ver}
The Veronese surface $V\subset\mathbb P^5$ is an arithmetically Cohen-Macaulay embedding which is not arithmetically Gorenstein. The study of vector bundles on $V$ is trivial if we view $V$ as $\mathbb P^2$ embedded with $\Oo_{\PP^2}(2)$. We consider the full exceptional collection $\langle E_0, E_1, E_2\rangle=\langle \Oo_{\PP^2}, \Oo_{\PP^2}(1), \Oo_{\PP^2}(2)\rangle$ and then its dual in reverse order $\Ee_{\bullet}$ consists of the following
$$\Ee_2=\Oo_{\PP^2}(-2)~,~\Ee_1=\Oo_{\PP^2}(-1)~,~ \Ee_0=\Oo_{\PP^2}.$$

\noindent Define $F_k=\Ff_k$ to be the vector bundle satisfying Condition (\ref{order}) in Theorem \ref{use} with respect to the prescribed collection $\langle E_0, E_1, E_2\rangle$, that is, 
$$
 \Ff_2=\Oo_{\PP^2}(-1)~,~ \Ff_1=\Omega_{\PP^2}^1(1)~,~  \Ff_0=\Oo_{\PP^2}. 
$$
In other words, the collection $\langle F_2, F_1, F_0\rangle$ is the right dual collection of $\langle E_0, E_1, E_2 \rangle$. 

Let $\Aa$ be a vector bundle on $V$ such that $\Aa(2)$ is Ulrich. Then we construct a Beilinson complex, quasi-isomorphic to $\Aa$, by calculating $H^i(\Aa\otimes \Ee_j)\otimes \Ff_j$ with  $i,j \in \{0,1,2\}$. We may assume due to \cite[Proposition 2.1]{ES} that
$$H^i(\Aa(-2))=H^i(\Aa)=0 \text{ for all $i$}$$
and then we can easily construct the Beilinson table as below; we put the collection $\Ff_{\bullet}$ on the top row and the collection $\Ee_{\bullet}$ on the bottom row
 \begin{center}
     \begin{tabular}{|c|c|c|}
       \hline
  $ \Oo_{\PP^2}(-1)$    & $\Omega_{\PP^2}^1(1)$ &$\Oo_{\PP^2}$  \\
       \hline
       \hline
       $0$ & $0$ & $0$  \\
       \hline
       $0$ & $a$ & $0$  \\
              \hline
       $0$& $0$& $0$  \\
        \hline
       \hline
       $\Oo_{\PP^2}(-2)$ & $\Oo_{\PP^2}(-1)$ & $\Oo_{\PP^2}$  \\
       \hline
     \end{tabular}
     \end{center}
We notice that the only nonzero entry $a:= h^1(\Aa(-1))$ appears on the diagonal. So the table is very simple and we can deduce that $\Omega_{\PP^2}^1(3)$ is the only indecomposable Ulrich bundle on $V$.
\end{example}

\begin{example}\label{qua}
Let $\Qq_4 $ be the smooth quadric hypersurface of $\mathbb P^{5}$. It has two distinct spinor bundles $\Sigma_1$ and $\Sigma_2$ of rank two given by the exact sequences
$$ 0 \to \Sigma_i ^\vee \rightarrow \Oo_{\Qq_4}^{\oplus 4}  \rightarrow \Sigma_{3-i} \to 0$$
for $i\in \{1,2\}$. As in Example \ref{ver} we may consider the full exceptional collection 
$$\langle E_0, \ldots, E_5\rangle=\langle \Oo_{\Qq_4},~\Oo_{\Qq_4}(1),~\Oo_{\Qq_4},~\Sigma_1^\vee(3),~\Sigma_2^\vee(3)[1],~\Oo_{\Qq_4}(3)[1]\rangle, $$
i.e. we have 
\begin{align*}
&\Ee_5=\Oo_{\Qq_4}(-3)~,~ \Ee_4=\Sigma_2(-3)~,~ \Ee_3=\Sigma_1(-3),\\
&\Ee_2=\Oo_{\Qq_4}(-2)~,~\Ee_1=\Oo_{\Qq_4}(-1)~,~ \Ee_0=\Oo_{\Qq_4}
\end{align*}
and $(k_0, \ldots, k_5)=(0,0,0,0,-1,-1)$. Then its right dual collection $\langle F_5,\ldots, F_0\rangle$ consists of the following vector bundles $F_i=\Ff_i$ for each $i$:
\begin{align*}
&\Ff_5=\Oo_{\Qq_4}(-1)~,~ \Ff_4=\Sigma_1(-1)~,~ \Ff_3=\Sigma_2(-1), \\
&\Ff_2=\Omega^2_{\PP^5}(2)_{|\Qq_4}~,~ \Ff_1=\Omega^1_{\PP^5}(1)_{|\Qq_4}~,~ \Ff_0=\Oo_{\Qq_4},
\end{align*}
due to the characterization in (\ref{order}). 

Let $\Aa$ be a vector bundle on $\Qq_4$ such that $\Aa(1)$ is Ulrich. As in Example \ref{ver} we construct a Beilinson complex, quasi-isomorphic to $\Aa$, by calculating $H^{i+k_j}(\Aa\otimes \Ee_j)\otimes \Ff_j$ with $i,j \in \{0,\dots 5\}$. Again we may assume due to \cite[Proposition 2.1]{ES} that
$$H^i(\Aa(-3))=H^i(\Aa(-2))=H^i(\Aa(-1))=H^i(\Aa)=0 \text{ for all $i$}.$$
Moreover there are no nonzero maps from $\Sigma_1(-1)$ to $\Sigma_2(-1)$. Thus we easily can construct the Beilinson table; we put the collection $\Ff_{\bullet}$ on the top row and the collection $\Ee_{\bullet}$ on the bottom row.
 \begin{center}
     \begin{tabular}{|c|c|c|c|c|c|}
       \hline
 $ \Oo_{\Qq_4}(-1)$& $\Sigma_1(-1)$ & $\Sigma_2(-1)$& $\Omega^2_{\PP^5}(2)_{|\Qq_4}$& $\Omega^1_{\PP^5}(1)_{|\Qq_4}$&$\Oo_{\Qq_4}$\\
       \hline
       \hline
       $ 0$& $0$ & $0$& $0$& $0$&$0$\\
       \hline
       $ 0$& $b$ & $0$& $0$& $0$&$0$\\
       \hline
       $ 0$& $0$ & $a$& $0$& $0$&$0$\\
       \hline
       $ 0$& $0$ & $0$& $0$& $0$&$0$\\
       \hline
       $ 0$& $0$ & $0$& $0$& $0$&$0$\\
       \hline
       $ 0$& $0$ & $0$& $0$& $0$&$0$\\
       \hline
       \hline
       $ \Oo_{\Qq_4}(-3)$& $\Sigma_2(-3)$ & $\Sigma_1(-3)$& $\Oo_{\Qq_4}(-2)$& $\Oo_{\Qq_4}(-1)$&$\Oo_{\Qq_4}$\\
       \hline
     \end{tabular}
     \end{center}
We notice that the only nonzero entries $a:= h^3(\Aa\otimes\Sigma_1(-3))$ and $b:=h^3(\Aa\otimes\Sigma_2(-3))$ appear on the diagonal. So the table is very simple and we can deduce that $\Sigma_1$ and $\Sigma_2$ are the only indecomposable Ulrich bundles on $V$.

Similarly for  smooth quadric hypersurfaces of any (even or odd) dimension we can obtain a table with nonzero entries only on the diagonal, proving that the indecomposable Ulrich bundles are the spinor bundles; see \cite{AO}.
\end{example}

In the next section we prove that for any rational normal scroll we can obtain a table with nonzero entries only on the diagonal; we cannot expect such a behavior for varieties of almost minimal degree. Let us see an example of a Del Pezzo threefold in the following.

\begin{example}\label{ver3}
The Veronese threefold $V\subset\mathbb P^9$ is Del Pezzo of degree $8$. The study of vector bundles on $V$ is trivial if we view $V$ as $\mathbb P^3$ embedded with $\Oo_{\PP^3}(2)$. We consider the full exceptional collection $\langle E_0, \ldots, E_3\rangle= \langle \Oo_{\PP^ 3}, \Oo_{\PP^3}(1), \Oo_{\PP^3}(2), \Oo_{\PP^3}(3)\rangle$ and then the collection $\Ee_{\bullet}$ in reverse order is given by 
$$
\Ee_3=\Oo_{\PP^3}(-3)~,~ \Ee_2=\Oo_{\PP^2}(-2)~,~\Ee_1=\Oo_{\PP^3}(-1)~,~ \Ee_0=\Oo_{\PP^3}$$
and the right dual collection is given by
$$
 \Ff_3=\Oo_{\PP^3}(-1)~,~ \Ff_2=\Omega_{\PP^3}^2(2)~,~\Ff_1=\Omega_{\PP^3}^1(1)~, ~ \Ff_0=\Oo_{\PP^3}.
$$

Now let $\Aa$ be a vector bundle on $V$ such that $\Aa(2)$ is Ulrich. Then we construct a Beilinson complex, quasi-isomorphic to $\Aa$, by calculating $H^i(\Aa\otimes \Ee_j)\otimes \Ff_j$ with  $i,j \in \{0,1,2,3\}$. We may assume due to \cite[Proposition 2.1]{ES} that
$$H^i(\Aa(-4))=H^i(\Aa(-2))=H^i(\Aa)=0 \text{ for all $i$}$$
and then we can easily construct the Beilinson table as below; we put the collection $\Ff_{\bullet}$ on the top row and the collection $\Ee_{\bullet}$ on the bottom row
 \begin{center}
     \begin{tabular}{|c|c|c|c|}
       \hline
  $ \Oo_{\PP^3}(-1)$  & $\Omega_{\PP^3}^2(2)$  & $\Omega_{\PP^3}^1(1)$ &$\Oo_{\PP^3}$  \\
       \hline
       \hline
       $0$ & $0$ & $0$ & $0$ \\
       \hline
       $b$ & $0$ & $0$ & $0$ \\
       \hline
       $0$ & $0$ & $a$ & $0$ \\
       \hline
       $0$ & $0$ & $0$ & $0$ \\

        \hline
       \hline
      $\Oo_{\PP^3}(-3)$ & $\Oo_{\PP^3}(-2)$ & $\Oo_{\PP^3}(-1)$ & $\Oo_{\PP^3}$  \\
       \hline
     \end{tabular}
     \end{center}
We notice that the only nonzero entries are $a:= h^1(\Aa(-1))$ which is on the diagonal and $b:= h^1(\Aa(-3))$ which is not on the diagonal.
\end{example}

\begin{remark}
Although we have a nonzero entry off the diagonal in Example \ref{ver3}, we may still use the Beilinson table to describe the Ulrich bundles on $V$. Indeed, any Ulrich bundle $\Ee$ on $V$ fits into the following exact sequence
$$0\to \Oo_{\PP^3}(1)^{\oplus b} \stackrel{\mathrm{M}}{\to} \Omega_{\PP^3}^1(3)^{\oplus a} \to \Ee \to 0.$$
By \cite[Lemma 2.4(iii)]{CH} we also get $a=b$ and so the map $\mathrm{M}=(s_{ij})$ is given as an $(a\times a)$-matrix with entries $s_{ij} \in H^0(\Omega_{\PP^3}^1(2)) \cong \CC^{\oplus 4}$. In particular, there is no Ulrich line bundle on $V$ and we get that the null-correlation bundles on $\PP^3$ twisted by $2$ are the only rank two Ulrich bundles on $V$. 
\end{remark}


\section{Classification on rational normal scroll}\label{sec3}

 Let $S=S(a_0, \ldots, a_n)$ be a smooth rational normal scroll, the image of $\PP (\Ee)$ via the morphism defined by $\Oo_{\PP (\Ee)}(1)$, where $\Ee \cong \oplus_{i=0}^n \Oo_{\PP^1}(a_i)$ is a vector bundle of rank $n+1$ on $\PP^1$ with $0< a_0 \le \ldots \le a_n$. Letting $\pi : \PP (\Ee) \rightarrow \PP^1$ be the projection, we may denote by $H$ and $F$, the hyperplane section corresponding to $\Oo_{\PP(\Ee)}(1)$ and the fibre corresponding to $\pi^*\Oo_{\PP^1}(1)$, respectively. Then we have $\Pic (S)\cong \ZZ\langle H,F\rangle$ and $\omega_S \cong \Oo_S(-(n+1)H+(c-2)F)$, where $c:=\sum_{i=0}^n a_i$ is the degree of $S$. We will simply denote $\Oo_S(aH+bF)$ by $\Oo_S(a+b,a)$, in particular, $\Oo_S(F)=\Oo_S(1,0)$. From now on we fix an ample line bundle on $S$ to be $\Oo_S(H)=\Oo_S(1,1)$.

For the computational purpose, we use the following lemma.
\begin{lemma}[\cite{EH}, see also \cite{MR}]\label{lem}
For any $i=0,\ldots,n+1$, we have
\begin{itemize}
\item [(i)] $H^i(S, \Oo_S(a+b,a)) \cong H^i(\PP^1, \mathrm{Sym}^a\Ee \otimes \Oo_{\PP^1}(b))$ if $a\ge 0$;
\item [(ii)] $H^i(S, \Oo_S(a+b,a)) =0$ if $-n-1<a<0$;
\item [(iii)] $H^i(S, \Oo_S(a+b,a)) \cong H^{n+1-i}(\PP^1, \mathrm{Sym}^{-a-n-1}\Ee \otimes \Oo_{\PP^1}(c-b-2))$ if $a\le -n-1$.
\end{itemize}
\end{lemma}

Recall the dual of the relative Euler exact sequence of $S$:
\begin{equation}\label{eq1}
0\to \Omega_{S|\PP^1}^1 (1,1) \to \Bb:=\oplus_{i=0}^{n}\Oo_S(a_i,0) \to \Oo_S(1,1) \to 0,
\end{equation}
and so we have $\omega_{S|\PP^1} \cong \Oo_S(-(n+1)H+cF)\cong \Oo_S(c-n-1, -n-1)$. The long exact sequence of exterior powers associated to (\ref{eq1}) is

\begin{equation}\label{eqq1}
\begin{split}
0\to &\Oo_S(c-n,-n) \to \wedge^n \Bb(-n+1, -n+1) \stackrel{d_{n-1}}{\to}\\ &\wedge^{n-1}\Bb(-n+2, -n+2) \stackrel{d_{n-2}}{\to} \cdots \stackrel{d_1}{\to} \Bb \to \Oo_S(1,1) \to 0.
\end{split}
\end{equation}
Now (\ref{eqq1}) splits into
\begin{equation}\label{eqqq1}
0\to \Omega_{S|\PP^1}^i (i,i)\to \wedge^i \Bb \to \Omega_{S|\PP^1}^{i-1}(i-1,i-1) \to 0
\end{equation}
for each $i=1,\ldots, n$, and we have $\mathrm{Im}(d_i\otimes \Oo_S(i-1,i-1))\cong \Omega_{S|\PP^1}^{i}(i,i)\subset \wedge^i \Bb$.

\begin{lemma}
\label{lemma:BuildingBlocks}\label{bb}
The bundle $\Omega_{S|\PP^1}^i(i,i+1)$ is Ulrich with slope $c-1$ with respect to $H=\Oo_S(1,1)$ for each $i=0,\dots, n$.
\end{lemma}
\begin{proof}
Consider the twist of the following truncation of (\ref{eqq1}) by $\Oo_S(i-1, i)$
\begin{equation}\label{hhy}
\begin{split}
0\to \Oo_S(c-n,-n) \to& \wedge^n \Bb(-n+1, -n+1) \to \cdots \\ &\to \wedge^{i+1}\Bb(-i, -i) \to \Omega_{S|\PP^1}^i \to 0.
\end{split}
\end{equation}
Applying Lemma \ref{lem} to (\ref{hhy}), we obtain that for any $k\in \ZZ$ and $j\in \{1,\cdots, n\}$
\begin{align*}
0=&h^{n-i+j}(\Oo_S(c-n+i+k-1, c-n+i+k))\\
=&h^{n-i-1+j}(\wedge^n \Bb (k-n+i,k-n+i+1))\\
=&\cdots = h^j(\wedge^{i+1}\Bb (k-1,k)),
\end{align*}
and it implies that $h^j(\Omega_{S|\PP^1}^i(k+i-1, k+i))=0$. Thus the sheaf $\Omega_{S|\PP^1}^i(i-1,i)$ is aCM and hence $\Omega_{S|\PP^1}^i(i,i+1)$ is also aCM. Now we may apply the same argument above to show
$$H^0(\wedge^{i+1}\Bb(0,1)) \cong H^0(\Omega_{S|\PP^1}^i(i,i+1))$$
whose dimension is $c{\times}{n \choose i} =\deg (S){\times} \mathrm{rank}(\Omega_{S|\PP^1}^i(i,i+1))$, and so the sheaf $\Omega_{S|\PP^1}^i(i,i+1)$ is Ulrich.

The slope computation is straightforward.
\end{proof}

\begin{remark}
In practice, instead of working with the Ulrich bundles $\Omega_{S|\PP^1}^i(i,i+1)$, we shall work with their twists by $\Oo_S(-H)$ i.e. with $\Omega_{S|\PP^1}^i(i-1,i)$. The reason will become clear in the main Theorem. The spectral sequence will be much simpler if we use bundles with no sections. The bundles $\Omega_{S|\PP^1}^i(i,i+1)$ will appear to be the building blocks for all the Ulrich bundles on $S$.
\end{remark}

%

\begin{lemma}\label{llem}
For any $i,j\in \{0,\ldots, n\}$ with $i\ne j$, we have
\begin{equation}
\mathrm{Hom}(\Omega_{S|\PP^1}^i (i,i), \Omega_{S|\PP^1}^j (j,j))=0.
\end{equation}
%
\end{lemma}
\begin{proof} We prove the vanishing for $i>j$ first. Denote by
\[
\mathcal H:=\mathcal{H}\emph{om}(\Omega_{S|\PP^1}^i (i,i), \Omega_{S|\PP^1}^j (j,j))\cong
\Omega_{S|\PP^1}^j\otimes (\Omega_{S|\PP^1}^i )^\vee (j-i,j-i).
\]
Similarly to the projective space case in \cite[Section 0.1]{NT}, the Koszul complex provides us with a natural sheaf morphism
\[
\alpha:\wedge^{i-j}\Bb^\vee\to \mathcal H.
\]

Indeed, by dualizing (\ref{eqqq1}) we obtain a short exact sequence:
\[
0\to (\Omega_{S|\PP^1}^{i-j-1})^\vee (j-i-1,j-i-1)\to\wedge^{i-j}\Bb^\vee\to (\Omega_{S|\PP^1}^{i-j})^\vee (j-i,j-i)\to 0
\]
and $\alpha$ is obtained from the surjection $\wedge^{i-j}\Bb^\vee\to (\Omega_{S|\PP^1}^{i-j})^\vee (i-j,i-j)$ and from the natural map $\Omega_{S|\PP^1}^i (i,i)\otimes (\Omega_{S|\PP^1}^{i-j})^\vee (j-i,j-i)\to \Omega_{S|\PP^1}^j (j,j)$.
For each point $x\in\mathbb P^1$, denote by $F_x$ the fibre over $x$. After restricting to $F_x$ and taking global sections, the map $\alpha$ induces the following usual natural isomorphism over the $n$--dimensional projective space; see \cite[Section 0.1]{NT}.
\begin{equation}
\label{eqn:isoPn}
H^0(F_x,(\wedge^{i-j}{\Bb^\vee})_{|F_x}) \cong \mathrm{Hom}_{F_x}(\Omega_{F_x}^i (i), \Omega_{F_x}^j (j))\cong  H^0(F_x,\mathcal H_{|{F_x}}).
\end{equation}

It follows by the Grauert theorem that $\pi_*(\wedge^{i-j}\Bb^\vee)$ and $\pi_*(\mathcal H)$ are vector bundles of the same rank. The morphism $\alpha$ induces a vector bundle morphism from $\pi_*(\wedge^{i-j}\Bb^\vee)$ to $\pi_*(\mathcal H)$ which is, by (\ref{eqn:isoPn}) and Grauert,  an isomorphism on each fibre (of the vector bundle), and hence it is an isomorphism. Therefore, we obtain an isomorphism
\[
H^0(\mathbb P^1,\pi_*(\wedge^{i-j}\Bb^\vee))\cong H^0(\mathbb P^1,\pi_*(\mathcal H))
\]
and, since $\wedge^{i-j}\Bb^\vee$ has no nontrivial global sections, it follows that $\mathcal H$ has no nonzero global sections.


In the case $i<j$, we use the well-known fact that on the projective space $\mathbb P^n$ there are no nontrivial morphisms from $\Omega^i_{\mathbb P^n}(i)$ to $\Omega^j_{\mathbb P^n}(j)$ for $i<j$. Then
\[
\mathrm{Hom}_{F_x}(\Omega_{F_x}^i (i), \Omega_{F_x}^j (j)=0
\]
and, by Grauert's Theorem, it follows that $\pi_*(\mathcal H)=0$ and hence
\[
\mathrm{Hom}(\Omega_{S|\PP^1}^i (i,i), \Omega_{S|\PP^1}^j (j,j))=0.
\]
\end{proof}

\begin{remark}
Although the sheaf $\Omega_{S|\PP^1}^i (i,i)$ is stable, as we shall see, it is not possible to use in the proof of Lemma \ref{llem} the fact that there is no nontrivial morphism from a stable sheaf to another stable sheaf with smaller slope, because the slope of $\Omega_{S|\PP^1}^i(i,i)$ with respect to the ample line bundle $\Oo_S(1,1)$ is zero for any $i$.
\end{remark}

Now we introduce suitable full exceptional collection; \cite[Corollary 2.7]{Orlov}, see also \cite{AB}.

\begin{example}\label{fec}
Consider a function $\sigma : \ZZ \rightarrow \ZZ^{\oplus 2}$ sending $i$ to the pair of integers $(\sigma_1,\sigma_2)$ such that $\sigma_1+\sigma_2=i$ and $\sigma_1-\sigma_2\in \{0,1\}$. For each integer $i\in \{0,\ldots, 2n+1\}$, define $\Ee_i'=\Oo_S(\sigma (-i+1))$ and $k_i=\sigma_2(-i+1)$ to have a collection 
$$\langle \Ee_{2n+1}'[k_{2n+1}], \ldots, \Ee_0'[k_0]\rangle.$$
For example, we have $\Ee'_{2n+1} [k_{2n+1}]=\Oo_S(-n,-n)[-n]$. Note that this collection is full for $D^b(S)$. Apply the left mutation to the pair $(\Ee'_1[k_1], \Ee'_0[k_0])=(\Oo_S, \Oo_S(1,0))$ to obtain another full exceptional collection
\begin{equation}
\langle \Ee_{2n+1}[k_{2n+1}], \ldots, \Ee_0[k_0] \rangle,
\end{equation}
where $\Ee_1=\Oo_S(-1,0)$ and $\Ee_0=\Oo_S$. For example, when $n=3$, the full exceptional collection $\Ee_{\bullet}[k_{\bullet}]$ consists of the following:
\begin{align*}
&\Ee_7[k_7]=\Oo_S(-3,-3)[-3],~ \Ee_6[k_6]=\Oo_S(-2,-3)[-3],~ \Ee_5[k_5]=\Oo_S(-2,-2)[-2], \\
&\Ee_4[k_4]=\Oo_S(-1,-2)[-2],~\Ee_3[k_3]=\Oo_S(-1,-1)[-1],~ \Ee_2[k_2]=\Oo_S(0,-1)[-1], \\
&\Ee_1[k_1]=\Oo_S(-1,0)~,~ \Ee_0[k_0]=\Oo_S.
\end{align*}
Indeed, the collection above is obtained by taking the dual in reverse order of the full exceptional collection $E_{\bullet}=\langle E_0, \ldots, E_{2n+1} \rangle$;
$$E_{\bullet}=\langle \Oo_S,~ \Oo_S(1,0),~ \Oo_S(0,1)[1],~ \Oo_S(1,1)[1],~ \Oo_S(1,2)[2],~ \ldots ,~ \Oo_S(n,n)[n]    \rangle,$$
which is also obtained by applying a right mutation to the first pair in the following full collection 
$$\langle \Oo_S(-1,0),~\Oo_S,~ \Oo_S(0,1)[1],~ \Oo_S(1,1)[1],~ \Oo_S(1,2)[2],~ \ldots ,~ \Oo_S(n,n)[n]    \rangle.$$
Now let $\langle F_{2n+1}, \ldots, F_0\rangle$ be the right dual collection of $E_{\bullet}$. Indeed it consists of the following vector bundles $F_i=\Ff_i$ for each $i$ due to (\ref{order})
\begin{align*}
&\Ff_{2n+1}=\Oo_S(c-3,-1)~,~ \Ff_{2n}=\Oo_S(c-2,-1), \\
&\Ff_{2n-1}=\Omega_{S|\PP^1}^{n-1}(n-3,n-1)~,~ \Ff_{2n-2}=\Omega_{S|\PP^1}^{n-1}(n-2,n-1), \dots ,\\
&\Ff_3=\Omega_{S|\PP^1}^1(-1,1)~,~ \Ff_2=\Omega_{S|\PP^1}^1(0,1)~,~ \Ff_1=\Oo_S(-1,0)~,~ \Ff_0=\Oo_S.
\end{align*}
For example, for $i=2n+1$, we have $H^{\bullet} (\Ee_{2n+1}[k_{2n+1}]\otimes \Ff_{2n+1}) \cong H^{2n+1}(\omega_S[-n])\cong \CC$.
\end{example}

\begin{theorem}\label{volon}
Let $\Vv$ be a vector bundle on $S$. Then $\Vv$ is Ulrich if and only if $\Vv$ admits a filtration
$$0=\Bb_0\subset \Bb_1\subset\cdots\subset\Bb_{n+1}=\Vv$$
with $\Bb_{i+1}/\Bb_i\cong \Omega_{S|\PP^1}^i(i,i+1)^{\oplus a_i}$ for some $a_i \in \NN$ and each $i$. Moreover, the bundles $\Omega_{S|\PP^1}^i$ are stable and the graded sheaf $\bigoplus_{i=0}^{n}\Omega_{S|\PP^1}^i(i,i+1)^{\oplus a_i}$ associated to this filtration on $\Vv$ coincides, up to a permutation of its factors, with the graded sheaf of the Jordan-H\"older filtration of $\Vv$.
\end{theorem}

\begin{proof}
We have verified that each of the bundles $\Omega_{S|\PP^1}^i(i,i+1)$ is Ulrich and so the converse implication is clear.

Now we prove the direct implication. The idea is to consider the Beilinson type spectral sequence associated to $\Aa:=\Vv(-1,-1)$ and identify the members of the graded sheaf associated to the induced filtration as the sheaves mentioned in the statement. We assume due to \cite[Proposition 2.1]{ES} that
$$H^i(\Aa(-j,-j))=0 \text{ for all $0\le i \le n+1$ and $0\le j \le n$}$$
and consider the full exceptional collection $\Ee_{\bullet}$ and left dual collection $\Ff_{\bullet}$ in Example \ref{fec}. We construct a Beilinson complex, quasi-isomorphic to $\Aa$, by calculating $H^{i+k_j}(\Aa\otimes \Ee_j)\otimes \Ff_j$ with  $i,j \in \{1, \ldots, 2n+2\}$ to get the following table. Here we use several vanishing in the intermediate cohomology of $\Aa, \Aa(-1,-1),\cdots,  \Aa(-n,-n)$ together with vanishing of cohomology $H^0$ and $H^{n+1}$:

\begin{center}\begin{tabular}{|c|c|c|c|c|c|c|c|}
\hline
$\Ff_{2n+1}$ &$\Ff_{2n}$&$\Ff_{2n-1}$& $\Ff_{2n-2}$& \dots
&$\Ff_2$&$\Ff_1$ &$\Ff_0$\\
\hline
\hline
$0$        &$0$        &$0$         &	 $0$	 	& \dots		 	& $0$		 	& $0$			& $0$	\\
$0$        &$H^{n}$    &$0$		    &	 $0$	& \dots		 	& $0$		 	& $0$			& $0$	\\
$0$        &$H^{n-1}$  &$0$	        &	 $H^n$	 	& \dots		    & $0$		    & $0$			& $0$	\\
$0$        &$H^{n-2}$  &$0$	        &	 $H^{n-1}$	& \dots		 	& $0$		 	& $0$			& $0$	\\
\vdots     &\vdots     &\vdots	    &	 \vdots	    & \vdots		& \vdots		& 	\vdots		& \vdots	\\
$0$        &$0$        &$0$		    &	 $H^1$	 	& \dots	 	    & $H^n$		 	& $0$		& $0$	\\
$0$        &$0$        &$0$		    &	 $0$	 	& \dots	 	    & $H^{n-1}$		 	& $H^n$			& $0$	\\
$0$        &$0$        &$0$		    &	 $0$	 	& \dots	 	    & $H^{n-2}$		 	& $H^{n-1}$		& $0$	\\
\vdots     &\vdots     &\vdots	    &	 \vdots	    & \vdots		& \vdots		& 	\vdots		& \vdots	\\
$0$        &$0$        &$0$		    &	 $0$	 	& \dots		 	& $H^1$		 	& $H^2$			& $0$	\\
$0$        &$0$        &$0$		    &	 $0$	 	& \dots		 	& $0$		 	& $H^1$			& $0$	\\
$0$        &$0$        &$0$		    &	 $0$	 	& \dots		 	& $0$		 	& $0$			& $0$	\\

\hline
\hline
$\Ee_{2n+1}$ &$\Ee_{2n}$ &$\Ee_{2n-1}$& $\Ee_{2n-2}$ & \dots
 &$\Ee_2$& $\Ee_1$ &$\Ee_0$\\
 \hline
\end{tabular}
\end{center}

\begin{lemma}\label{ttt}
For any $i,j \in \{1\} \cup \{2,4,\ldots, 2n\}$ with $i>j$, we have
$$\mathrm{Hom}(\Ff_i, \Ff_j)=0.$$
\end{lemma}
\begin{proof}
By Lemma \ref{llem}, it is enough to consider the case when $i=2n$ and in this case we have $\mathrm{Hom}(\Ff_{2n}, \Ff_j) \cong H^0(\Ff_j (2-c,1))$. If $j=1$, then it is isomorphic to $H^0(\Oo_S(1-c,1)) \cong H^0(\PP^1, \Ee \otimes \Oo_{\PP^1}(-c))$ by Lemma \ref{lem} and it is trivial. For $j>1$, we need to show that $H^0(\Omega_{S|\PP^1}^i (i+1-c, i+1))=0$. Consider the twist of (\ref{hhy}) by $\Oo_S(i+1-c, i+1)$ and apply Lemma \ref{lem} to get that
\begin{align*}
h^{n-i}(\Oo_S(i+1-n, i+1-n)&=h^{n-i-1}(\wedge^n \Bb \otimes \Oo_S(i+2-c-n, i+2-n))\\
&=\cdots = h^0(\wedge^{i+1}\Bb \otimes \Oo_S(1-c,1))=0
\end{align*}
and it implies the vanishing we want.
\end{proof}

\noindent By Lemma \ref{ttt}, we may conclude that all the entries off the diagonal must be zero and thus we get

\begin{center}\begin{tabular}{|c|c|c|c|c|c|c|c|}
\hline
$\Ff_{2n+1}$ &$\Ff_{2n}$&$\Ff_{2n-1}$& $\Ff_{2n-2}$& \dots
&$\Ff_2$&$\Ff_1$ &$\Ff_0$\\
\hline
\hline

$0$        &$0$        &$0$         &	 $0$	 	& \dots		 	& $0$		 	& $0$			& $0$	\\
$0$        &$b$    &$0$		    &	 $0$	& \dots		 	& $0$		 	& $0$			& $0$	\\
$0$        &$0$  &$0$	        &	 $0$	 	& \dots		    & $0$		    & $0$			& $0$	\\
$0$        &$0$  &$0$	        &	 $H^{n-1}$	& \dots		 	& $0$		 	& $0$			& $0$	\\
\vdots     &\vdots     &\vdots	    &	 \vdots	    & \vdots		& \vdots		& 	\vdots		& \vdots	\\
$0$        &$0$        &$0$		    &	 $0$	 	& \dots	 	    & $0$		 	& $0$		& $0$	\\
$0$        &$0$        &$0$		    &	 $0$	 	& \dots	 	    & $0$		 	& $0$			& $0$	\\
$0$        &$0$        &$0$		    &	 $0$	 	& \dots	 	    & $0$		 	& $0$		& $0$	\\
\vdots     &\vdots     &\vdots	    &	 \vdots	    & \vdots		& \vdots		& 	\vdots		& \vdots	\\
$0$        &$0$        &$0$		    &	 $0$	 	& \dots		 	& $H^1$		 	& $0$			& $0$	\\
$0$        &$0$        &$0$		    &	 $0$	 	& \dots		 	& $0$		 	& $a$			& $0$	\\
$0$        &$0$        &$0$		    &	 $0$	 	& \dots		 	& $0$		 	& $0$			& $0$	\\

\hline
\hline
$\Ee_{2n+1}$ &$\Ee_{2n}$ &$\Ee_{2n-1}$& $\Ee_{2n-2}$ & \dots
 &$\Ee_2$& $\Ee_1$ &$\Ee_0$\\
 \hline
\end{tabular}
\end{center}
where $a:=h^1(\Aa(-1,0))$ and $b:=h^n(\Aa(-n+1,-n))$. This yields to the desired filtration.

In order to verify the last assertion, it suffices to check the stability of the bundles $\Omega^i_{S|\mathbb P^1}$ for all $i$. Assume that $\Omega^i_{S|\mathbb P^1}$ was not stable for some $i$. Then, since $\Omega^i_{S|\mathbb P^1}(i,i+1)$ is Ulrich and so semistable, it will contain a stable Ulrich subbundle, say $\mathcal F$, with the same slope $c-1$. Applying the first part of the result, we obtain a filtration on $\mathcal F$ whose terms have all the same slope $c-1$, and, in particular, from the stability, the bundle $\mathcal F$ must be isomorphic to $\Omega^j_{S|\mathbb P^1}(j,j+1)$ for some $j$. Therefore we obtain an injective morphism from $\Omega^j_{S|\mathbb P^1}(j,j+1)$ to $\Omega^i_{S|\mathbb P^1}(i,i+1)$. Now it follows from Lemma \ref{llem} that $i=j$ and so $\mathcal F\cong \Omega^i_{S|\mathbb P^1}(i,i+1)$, which implies the stability of $\Omega^i_{S|\mathbb P^1}$.
\end{proof}

\begin{remark}\label{con}
A consequence of the main result is that the moduli spaces of Ulrich bundles are zero-dimensional. This follows from the fact that when we construct moduli spaces of semistable bundles, we work with $S$--equivalence classes.
\end{remark}

\begin{remark}
It follows directly from Theorem \ref{volon} that, for $n\ge 2$, the only stable Ulrich bundles of rank $\le n-1$ are $\Oo_S(0,1)$ and $\Oo_S(c-1,0)$. In particular, if $n\ge 3$, there are no stable Ulrich bundles of rank two on $S$ and all the strictly semistable Ulrich bundles of rank two are extensions
\[
0\to \Oo_S(0,1)\to \Vv\to \Oo_S(c-1,0)\to 0.
\]
In the case $n=2$ the only stable Ulrich bundle is $\Omega^1_{S|\mathbb P^1}(1,2)$.
\end{remark}

\begin{remark}
In surface case, the previous result is compatible with \cite[Theorem 2.1]{ACM}. Indeed, for the bundle $\Ee=\Oo_{\PP^1}(\beta)\oplus\Oo_{\PP^1}(\beta -e)$ on $\mathbb P^1$ with $\beta>e>0$, the only Ulrich line bundles are
$$\Oo_S(c-1,0)=\Oo_S((2\beta-e-1)F) \text{ and } \Oo_S(0,1)=\Oo_S(C_0+(\beta-1)F),$$
where $C_0$ is the minimal section.

All the other Ulrich bundles arise from extension of the form
$$
0\to \Oo_S(0,1)^{\oplus a}\to \Vv\to \Oo_S(c-1,0)^{\oplus b} \to 0
$$
and hence, there are no stable Ulrich bundles of rank $\ge 2$. This was previously proven in \cite{FM}.
\end{remark}

\begin{remark}
For $n>2$ we may compute
$$\Ext_S^1(\Omega_{S|\PP^1}^1(0,1),\Oo_S(-1,0))=H^1(\Omega_{S|\PP^1}^1(1,1)^\vee),$$
and from the dual of (\ref{eq1}) we get
$$h^1(\Omega_{S|\PP^1}^1(1,1)^\vee)=h^1(\Oo_S(-a_i,0))=0$$
if and only if $a_i=1$ for any $i$, so in this case the filtration can be simpler. Now take $(a_0, \ldots, a_n)=(1,\dots,1)$, i.e. $S=S(1,\dots,1)\cong \PP^1 \times \PP^n$. Then it is embedded into $\PP^{2n+1}$ by the complete linear system $|\Oo_S(1,1)|$ as a Segre variety. Then we have $\Omega_{S|\PP^1}^1 \cong \Oo_{\PP^1}\boxtimes \Omega_{\PP^n}^1$ and so $\Omega_{S|\PP^1}^i \cong \Oo_{\PP^1}\boxtimes \Omega_{\PP^n}^i$. Using the notation above, we get
$$
\Ext_S^1(\Omega_{S|\PP^1}^{i}(i-1,i),\Omega_{S|\PP^1}^j(j-1,j))=H^1(\Oo_{\PP^1}(j-i)\boxtimes \Omega_{\PP^n}^j(j)\otimes\Omega_{\PP^n}^{i}(i)^\vee)
$$
which is zero for each $i\le j+1$ and equals
\[
H^0(\Oo_{\mathbb P^1}(i-j-2))^\vee \otimes \wedge^{i-j}\mathbb C^{n+1}
\]
for $i\ge j+2$. 
\end{remark}

\begin{example}
On $S=\PP^1 \times \PP^2$, a bundle $\Vv$ is Ulrich if and only if it admits a filtration $0=\Bb_0\subset \Bb_1\subset\Bb_2\subset\Bb_{3}=\Vv$ with
$$\Bb_{1}\cong\Oo_S(0,1)^{\oplus a_0}~,~\Bb_{2}/\Bb_1\cong \Oo_{\PP^1}(1)\boxtimes \Omega_{\PP^2}^1(2)^{\oplus a_1}~,~\Bb_{3}/\Bb_2\cong \Oo_S(3,0)^{\oplus a_2}$$ 
for some $a_0,a_1,a_2 \in \NN$. So the only stable Ulrich bundle of rank $\ge 2$ is $\Vv\cong\Oo_{\PP^1}(1)\boxtimes \Omega_{\PP^2}^1(2)$ and all the other indecomposable Ulrich bundles arise from the extensions
 \begin{equation}\label{Exn}
0\to \Oo_S(0,1)^{\oplus a_0}\to \Vv\to \Oo_S(3,0)^{\oplus a_2} \to 0.
\end{equation}
This case has been studied in more details in \cite{CFMS}, where all the aCM bundles are classified.
\end{example}

\begin{example}
On $S=\PP^1 \times \PP^3$, a bundle $\Vv$ is Ulrich if and only if it admits a filtration $0=\Bb_0\subset \Bb_1\subset\Bb_2\subset\Bb_{3}\subset\Bb_4=\Vv$ with 
\begin{align*}
&\Bb_{1}\cong\Oo_S(0,1)^{\oplus a_0}~,~\Bb_{2}/\Bb_1\cong \Oo_{\PP^1}(1)\boxtimes \Omega_{\PP^3}^1(2)^{\oplus a_1},\\
&\Bb_{3}/\Bb_2\cong\Oo_{\PP^1}(2)\boxtimes \Omega_{\PP^3}^2(3)^{\oplus a_2}~,~\Bb_{4}/\Bb_3\cong \Oo_S(4,0)^{\oplus a_3}
\end{align*}
for some $a_0,a_1,a_2,a_3 \in \NN$. So all the Ulrich bundles arise from the extensions
$$
0\to \Oo_S(0,1)^{\oplus a_0}\oplus\Oo_{\PP^1}(1)\boxtimes \Omega_{\PP^3}^1(2)^{\oplus a_1}\to \Vv\to\Oo_{\PP^1}(2)\boxtimes \Omega_{\PP^3}^2(3)^{\oplus a_2}\oplus \Oo_S(4,0)^{\oplus a_3} \to 0.
$$

\end{example}




\providecommand{\bysame}{\leavevmode\hbox to3em{\hrulefill}\thinspace}
\providecommand{\MR}{\relax\ifhmode\unskip\space\fi MR }
\providecommand{\MRhref}[2]{%
  \href{http://www.ams.org/mathscinet-getitem?mr=#1}{#2}
}
\providecommand{\href}[2]{#2}

\end{document}